\documentclass[11pt, reqno]{article}
\usepackage[margin=1in]{geometry}
\usepackage{hyperref}

\usepackage{mathtools}
\usepackage{amssymb}
\usepackage{amsthm}

\newtheorem{thm}{Theorem}

\newtheorem{lem}[thm]{Lemma}
\newtheorem{cor}[thm]{Corollary}
\newtheorem{con}[thm]{Conjecture}

\numberwithin{thm}{section}

\newcommand{\N}{\mathbb{N}}
\def\L{\mathcal{L}}
\def\C{\mathcal{C}}
\def\O{\mathcal{O}}

\title{Towards a Complete Local-Global Principle}
\author{Nikola Kuzmanovski}
\date{\vspace{-5ex}}

\begin{document}
\maketitle
\begin{abstract}
	Ahlswede and Cai proved that if a simple graph has nested solutions (NS) under the edge-isoperimetric problems, 
	and the lexicographic (lex) order produces NS for its second cartesian power,
	then the lex order produces NS for any finite cartesian power.
	Under very general assumptions, 
	we prove that if a graph and its second cartesian power have NS,
	then so does any finite cartesian power.
	Harper asked if this is true without any restriction.
	We also conjecture that it is.
	All graphs studied in the literature for which the lex order is optimal are regular.
	This lead Bezrukov and Elsässer to conjecture that if the lex order is optimal for the second cartesian power,
	then the original graph is regular.
	A counterexample to this conjecture is provided.
\end{abstract}

\section{Introduction}

Discrete isoperimetric inequalities have been studied for applications to pure mathematics and the sciences,
and for their own sake.
Harper solved the edge-isoperimetric problem on hypercubes \cite{HarperEIP1964} in order to partially prove Posner's wirelength conjecture.
Lindsay \cite{LindseyEIP1964} solved the problem on Hamming graphs and completely settled Posner's conjecture.
Harper's result has been rediscovered multiple times \cite{Bernstein1967, Clements1971, HartEIP1976, KleitmanKriegerRothschild1971}
In more recent times,
the solutions to the edge-isoperimetric problem on the Petersen graph was motivated by applications to multiprocessing computers.
The literature on applications is vast,
we point the reader to \cite{Bezrukov1999,HarperBook2004} for more details.

It is known \cite{EngelBook1997, HarperBook2004} that the Kruskal-Katona Theorem \cite{Katona1966, Kruskal1963} implies Harper's Theorem.
This connection holds in general for Macaulay posets and corresponding graphs \cite{Bezrukov1999Equiv}.
In \cite{BezrukovPiotrowskiPfaff2004} this connection was used to study Macaulay posets.

In this paper, we consider two graphs $G=(V_G,E_G)$ and $H=(V_H,E_H)$.
All graphs are simple.
Indices will be omitted when the graphs are clear from context.
For $A,B\subseteq V_G$ we define:
\begin{align*}
	I_G(A,B)&=\{\{u,v\}\in E_G\bigm| u\in A,\; v\in B\}, \\ 
	I_G(A)&= I_G(A,A), \\
	I_G(m)&= \max_{A\subseteq V_G,\; |A|=m} |I_G(A)|, \\
	\Theta_G(A) &= \{ \{u,v\} \in E_G \bigm | u\in A, v\not \in A\},\\
	\Theta_G (m) &= \min_{A\subseteq V_G,\; |A|=m} |\Theta_G(A)|.
\end{align*}
There are two classical \textit{edge-isoperimetric problems}.

\textit{The Boundary Problem:} for $m\in [|V_G|]$, 
find $A\subseteq V_G$ with $|A|=m$ and $|\Theta(A)| = \Theta(m)$.

\textit{The Induced Edges Problem:} for $m\in [|V_G|]$, 
find  $A\subseteq V_G$ with $|A|=m$ and $|I(A)|=I(m)$.

We call such sets $A$ \textit{optimal}. 
For regular graphs these two problems are equivalent,
which follows by the folklore result below (a proof can be found in \cite{HarperBook2004}).

\begin{lem}
	If $G=(V,E)$ is a regular graph with degree $n$ and $A\subseteq V$ then
	\begin{eqnarray*}
		|\Theta(A)| + 2|I(A)| = n|A|.
	\end{eqnarray*}
\end{lem}

The \textit{Cartesian product} of $H$ and $G$ is the graph $H\square  G$,
with the vertex set $V_H\times V_G$, 
whose two vertices
$(x,y)$ and $(u,v)$ are adjacent iff either $x=u$ and $\{y,v\}\in E_G$, or
$\{x,u\}\in E_H$ and $y=v$. 
The graph $G^d=G\square G\square\cdots \square G$ ($d$ times) is called the $d^{th}$ {\em Cartesian power} of $G$.

We say that $G$ has \textit{nested solutions (NS)}  if there are sets $A_1\subseteq A_2\subseteq \cdots \subseteq A_{|V|}$, 
such that for all $i\in \{1,\dots, |V|\}$ we have that $|A_i|=i$ and $A_i$ is optimal in the induced edges problem.
Notice that the NS property forces a total order on $V$.
Thus, if a graph has NS,
then without loss of generality we assume $V=\{0,1,\dots, |V|-1\}$, 
where each $\{0,1,\dots, k-1\}$ is optimal for $k\leq |V|$.
We call $\{0,1,\dots, k-1\}$ an \textit{initial segment} of size $k$ in $G$.
Total orders on $V$ that emerge from NS are called \textit{optimal orders}.

\textit{Lexicographic order}, $\L$ on $n$-tuples with integer entries,
is defined such that $(x_1,\dots,x_n) <_{\L} (y_1,\dots,y_n)$ 
iff there exists an index $i$, $1\leq i\leq n$, 
such that $x_j=y_j$ for $j<i$ and $x_i< y_i$.

{\it Colexicographic order}, $\C$ on $n$-tuples with integer entries,
is defined such that $(x_1,\dots, x_n) <_{\C} (y_1,\dots, y_n)$
iff there exists an index $i$, $1\leq i \leq n$, 
such that $x_j=y_j$ for $j> i$ and $x_i < y_i$.

These orders are important to mathematics.
They are involved in the solutions of many discrete extremal problems.
They are also important to other areas of mathematics and studying them is of great importance.

Macaulay's lex ideal theorem \cite{Macaulay1927} uses the lexicographic order to give lower bounds on Hilbert functions in polynomial rings.
This theorem eventually led to the discovery of Gröbner basis and the rise of computational commutative algebra.
Many generalization have been given in different rings \cite{Mermin2006, MerminMurai2010, MerminPeeva2006, MerminPeeva2007}.
Recently, the theory of Hilbert functions was linked with the theory of Macaulay posets \cite{KuzmanovskiMacaulay2023} .
It would be interesting to see results on Hilbert functions deduced from edge-isoperimetric inequalities in the future,
since edge-isoperimetric inequalities were already used to study Macaulay posets in \cite{BezrukovPiotrowskiPfaff2004}.

\begin{thm}[Ahlswede-Cai \cite{AhlswedeCai1997_II}, 1997]\label{lgp}
	If $G$ has NS and $\L$ is optimal for $G^2$ then $\L$ is optimal for
	$G^n$ for any $n\geq 3$.
\end{thm}

In \cite{AhlswedeCai1997_II}, 
Theorem \ref{lgp} was stated using generalized edge-isoperimetric functions,
but it has been mostly used for graphs.
Notice that Theorems \ref{lgp} holds if $\L$ is replaced with $\C$.
Ahlswede and Cai called Theorem \ref{lgp} a local-global principle.
The local-global principle was first generalized by Harper \cite{HarperBook2004}
to allow different graphs in the product.
Recently, in \cite{BezrukovKuzmanovskiLim2023} it was shown that the local-global principle holds for many different orders.

For $1 \leq m \leq |V|$ we define the $m$-th $\delta$-\textit{entry} of $G$ by $\delta_G(m) = I_G(m) - I_G(m-1)$, 
and we define the $\delta$\textit{-sequence} of $G$ to be $\delta_G = (\delta_G(1),\delta_G(2),\dots,\delta_G(|V|))$.
Some examples of $\delta$-sequences are:
\begin{align*}
	\delta_{K_n} &= (0,1,2,\dots, n-1)\\
	\delta_{\text{tree on $n$ vertices}} &= (0,1, 1, 1,\dots, 1)\\
	\delta_{\text{Petersen}} &= (0,1,1,1,2,1,2,2,2,3).
\end{align*}

\begin{lem}\label{lem1}
	If $G$ has NS
	then $\delta(i+1) -\delta(i)\leq 1$ for all $i\in \{1,\dots ,|V| - 1\}$.
\end{lem}

This result tells us that $\delta_G$ can be partitioned into (strictly) increasing monotonic segments. From our previous examples, $\delta_{K_n}$ has one monotonic segment,
$\delta_{\text{Petersen}}$ has $6$ monotonic segments 
and $\delta_{\text{tree on $n$ vertices}}$ has $n-1$ monotonic segments. 
By $M_{G,i}$ with $i\geq 1$ denote the set of vertices
corresponding to the entries of the $i$-th monotonic segment of $\delta_G$. 

\begin{align*}
	s_{G,i} &= \delta_G \left( 1 + \sum_{k=1}^{i-1}|M_{G,k}|\right),
\end{align*}
Informally, $s_i$ is starting value of the $i$-th monotonic segment.
Observe that $\delta(1)=s_1=0$.
We say that $G$ is $\delta$-\textit{dense} iff $s_{G,2},\dots, s_{G,r} > 1$.
Note that $K_n$ is $\delta$-dense,
but the Petersen graph and any tree on at least three vertices are not.
Our main result is the following theorem.

\begin{thm}\label{abstractLGP}
	Suppose that $G$ is $\delta$-dense.
	If $G$ and $G^2$ have NS then so does $G^d$ for all $d\geq 3$.
\end{thm}

Many results in the literature that solve the problem for $G^2$ (when it is not trivial) assume that $G$ is regular\cite{AhlswedeCai1997_II, BezrukovBulatovicKuzmanovski2018, BezrukovDasElsasser2000, BezrukovElsasser2003, Carlson2002}.
This led to the following conjecture.

\begin{con}\label{LexReg}
	{\rm (Bezrukov-Elsässer \cite{BezrukovElsasser2003}, 2003)}
	If $\L$ is optimal for $G^2$ then $G$ is regular.
\end{con}

In Section \ref{uniquenessSection} we prove Theorem \ref{abstractLGP}.
In Section \ref{concluding} we provide a counterexample to Conjecture \ref{LexReg}, 
and give some ideas on how to drop the $\delta$-dense condition in Theorem \ref{abstractLGP}.

\section{Uniqueness of $\L$ and $\C$ in $G^2$}\label{uniquenessSection}

Suppose that $G$ and $H$ admit optimal orders.
For $A\subseteq V_H\times V_G$ denote $A_i(a)=\{(x_1,x_2)\in A\bigm| x_i=a\}$. 
We say that $A$ is \textit{compressed} if 
$A_i(a) = \{0, 1, . . . , |A_i(a)|-1\}$ for $i = 1,2$ and any $a$.
A proof of the following folklore result can be found in \cite{HarperBook2004}.

\begin{lem}
	If $A_1 \subseteq \cdots \subseteq A_n\subseteq V_{H\square G}$, 
	then there exist compressed sets $A_1' \subseteq \cdots \subseteq A_n'$ such that $|A_i| = |A_i'|$ and $|I_{H\square G}(A_i)| \leq |I_{H\square G}(A_i')|$,
	for all $i\in [n]$.
\end{lem}

Thus, if $H\square G$ has NS, then $H\square G$ has NS $A_1\subseteq \cdots \subseteq A_{|V|}$ such that each $A_i$ is compressed.
These NS give an order on $H\square G$ 
such that each initial segment is a compressed set.
We will call such orders \textit{compressed optimal orders}.
Note that $\L$ and $\C$ are compressed optimal orders.
A very useful insight for the proof of Theorem \ref{uniqueness} is the following result.

\begin{lem}[Bezrukov \cite{Bezrukov1999Equiv} ,1999]
	If $A\subseteq V_{H\square G}$ is compressed then
	\begin{align*}
		|I(A)| = \sum_{(x,y)\in A} (\delta_H(x) + \delta_G(y)).
	\end{align*}
\end{lem}

\begin{thm}\label{uniqueness}
	Suppose that $G$ and $G^2$ have NS,
	and let $\O$ be a compressed optimal order on $G^2$.
	If $G$ is $\delta$-dense then $\O = \L$ or $\O = \C$.
\end{thm}
\begin{proof}
	Take a chain of compressed optimal sets $A_1\subseteq A_2\subseteq \cdots \subseteq A_{|V_H||V_G|}$,
	such that $|A_i|=i$.
	We always have $A_1 = \{(0,0)\}$, since $A_1$ is compressed.
	The claim is trivial for $|V_G|=1$, so suppose that $|V_G|>1$.
	Also, because $A_2$ is compressed we have $A_2=\{(0,0),(0,1)\}$ or $A_2=\{(0,0),(1,0)\}$.
	If $A_2=\{(0,0),(0,1)\}$ then Lemma \ref{firstColumn} and Lemma \ref{afterFirstColumn} give us $\O=\L$.
	If $A_2=\{(0,0),(1,0)\}$ then symmetry and the same argument gives us $\O=\C$.
	We will now prove these lemmas.
	The following inequality is useful for proving Lemma \ref{firstColumn} and Lemma \ref{afterFirstColumn}.
	
	\begin{lem}
		If $k\in \{3,\dots, |V_G|\}$ then $\delta(k) - \delta(2) > 0$.
	\end{lem}
	\begin{proof}
		Since $G$ is $\delta$-dense we have $\delta_G(k) - \delta_{G}(2)\geq \min_{n\geq 2} s_{G,n} - \delta_G(2) \geq 2-1>0$.
	\end{proof}
	
	\begin{lem}\label{firstColumn}
		If $A_2=\{(0,0),(0,1)\}$ then $A_{|V_G|} = \{(0,0), (0,1),\dots, (0, |V_G|-1)\}$.
	\end{lem}
	\begin{proof}
		Assume to the contrary that this is not the case.
		Then there is some $k\geq 3$ such that $A_k \neq \{(0,0),\dots, (0,k-1)\}$,
		and for all $k'<k$ we have $A_{k'} = \{(0,0),\dots, (0,k'-1)\}$.
		Thus, $A_k = \{(0,0),\dots, (0,k-2), (1,0)\}$,
		since $A_k$ is compressed and $A_{k-1} = \{(0,0),\dots, (0,k-2)\}$.
		However now we have a contradiction with the optimality of $A_k$, 
		\begin{align*}
			|I(\{(0,0), (0,1),\dots, (0,k-1)\})|-|I(A_k)| = \delta_G(k)-\delta_G(2) > 0,
		\end{align*}
	\end{proof}
	
	\begin{lem}\label{afterFirstColumn}
		If $A_{|V_G|} = \{(0,0),(0,1),\dots, (0, |V_G|-1)\}$ then $\O = \L$.
	\end{lem}
	\begin{proof}
		Assume for the purposes of a contradiction that the claim does not hold.
		Let $k$ be minimal such that $A_k$ is not an initial segment of the lex order.
		Also, let $(x,y)$ be the unique vertex in $A_{k-1}\setminus A_{k-2}$,
		and $(w,z)$ be the unique vertex in $A_{k}\setminus A_{k-1}$.
		We must have $y< |V_G|-1$, 
		since otherwise $A_{k-1}$ is an initial segment of lex and $A_k$ is compressed.
		This forces $(w,z) = (x+1,0)$.
		
		\textit{Case 1:} If $y>0$ then we get a contradiction with the optimality of $A_k$,
		\begin{align*}
			|I(A_{k-1}\cup \{(x,y+1)\})| - |I(A_k)| &= \delta(x+1) + \delta(y+2) - \delta(x+2) - \delta(1),\\
			&= \delta(x+1) + \delta(y+2) - \delta(x+2),\\
			&\geq \delta(x+1) + \delta(y+2) - \delta(x+1) - 1,\\
			&= \delta(y+2) - \delta(2),\\
			&>0.
		\end{align*}
		
		\textit{Case 2:} Suppose $y=0$. 
		So, there exist $t\geq 1$ such that for $s=k+t$ we have
		\begin{align*}
			A_s = A_{k-1} \cup \{(x+1,0), (x+2,0),\dots, (x+t,0)\} \cup \{(x,1)\}.
		\end{align*}
		Let 
		\begin{align*}
			C &= \{(x,0),(x,1),\dots, (x,t)\},\\
			R &= \{(x,0), (x+1,0), \dots, (x+t,0)\},\\
			B &= A_{k-2}\cup C\cup \{(x,t+1)\}.
		\end{align*}
		One has,
		\begin{align*}
			|I(B\setminus \{(x,t+1)\})| - |I(A_s\setminus \{(x,1)\})|
			&=|I(C)|-|I(R)| + |I(A_{k-2}, C)| - |I(A_{k-2}, R)|,\\
			&\geq |I(A_{k-2}, C)| - |I(A_{k-2}, R)|,\\
			&= (t+1)|I(A_{k-2},\{(x,0)\})| - \sum_{j=0}^{t} |I(A_{k-2}, \{(x+j,0)\})|\\
			&\geq (t+1)|I(A_{k-2},\{(x,0)\})| - \sum_{j=0}^{t} |I(A_{k-2},\{(x,0)\})|\\
			&=0.
		\end{align*}
		Thus, we a contradiction with the optimality of $A_s$, since from the above inequality we have
		\begin{align*}
			|I(B)| - |I(A_s)|
			&= \delta(x+1) + \delta(t+2) - \delta(x+1) - \delta(2) \geq \delta(t+2) - \delta(2) >0,
		\end{align*}
		
		Therefore, the claim holds, since in both cases we get a contradiction.
	\end{proof}
	
	As mentioned at that start of the proof we are done after the proof of lemmas \ref{firstColumn} and \ref{afterFirstColumn}.
\end{proof}

We can now prove Theorem \ref{abstractLGP}.

\begin{cor}
	Suppose that $G$ is $\delta$-dense.
	If $G$ and $G^2$ have NS then so does $G^d$ for all $d\geq 3$.
\end{cor}
\begin{proof}
	There exists a compressed optimal order on $G^2$, since $G^2$ has NS.
	Hence, Theorem \ref{uniqueness} forces this order to be $\L$ or $\C$.
	Therefore, we are done by Theorem \ref{lgp}.
\end{proof}

\section{Towards a Complete Local-Global Principle}\label{concluding}

We will first provide a counterexample to Conjecture \ref{LexReg}.
Let
\begin{align*}
	X &= (  \{x_0,x_1,x_2,x_3,x_4\},      \{\{x_0,x_1\}, \{x_1,x_2\}, \{x_3,x_4\}  \}), 
	\mbox{ two disjoint paths of sizes $3$ and $2$},\\
	Y &= (  \{y_0,y_1,y_2,y_3,y_4,y_5\},      \{\{y_0,y_1\}, \{y_0,y_2\},\{y_1,y_2\},\{y_3,y_4\},\{y_3,y_5\}, \{y_4,y_5\}  \}),\\ 
	&\mbox{ two disjoint cliques of size $3$}.
\end{align*}
Then for all $n\in \N$ define $Z_n = X \ast Y_1\ast \cdots \ast Y_{n-1}$,
It is easy to check that
\begin{align*}
	\delta_{Z_2} = (0,1,2,3,4,5,6,7,7,6,7,8,9,10,11,12,13).
\end{align*}

It is also easy to check that $\L$ is optimal for $Z_2^2$.
We say that $\delta_G$ is \textit{symmetric} iff
$\delta(i)+ \delta(|V| - i +1) = \delta(|V|)$ for all $i\in \{1,\dots, |V|\}$.
Notice that $Z_2$ is not regular because of the following result.

\begin{thm}\label{reg}
	{\rm (Bonnet, Sykora \cite{BonnetSikora2016})}
	$\delta_G$ is symmetric iff $G$ is regular.
\end{thm}

\begin{con}\label{newLex}
	For all $n\geq 1$ we have that $Z_n^d$ has NS.
\end{con}

Proving Conjecture \ref{newLex} for $d=1$ should follow from a similar argument to the one used in the proof of Theorem 5 in \cite{BezrukovBulatovicKuzmanovski2018}.
The case $d\geq 3$ will follows from Theorem \ref{lgp} if the case $d=2$ is handled.
Note that Theorem \ref{uniqueness} implies that $\L$ and $\C$ are the only compressed optimal orders.
The case $d=2$ is interesting because $Z_n$ is not regular.

All results point to Conjecture \ref{fullLGP}.
Note that in \cite{BezrukovKuzmanovskiLim2023},
Theorem \ref{lgp} was generalized to include the Petersen graph and many other cases.
The case for the product of tree was not handled in \cite{BezrukovKuzmanovskiLim2023},
but the authors conjecture that it can be included and suggest a way to prove this.

\begin{con}[Harper \cite{HarperBook2004}, 2004]\label{fullLGP}
	If $G$ and $G^2$ have NS then so does $G^d$ for all $d\geq 3$.
\end{con}

To solve Conjecture \ref{fullLGP} we suggest the following steps:
\begin{enumerate}
	\item\label{delta=1} Find all optimal orders when $\delta(i) \geq 1$.
	See \cite{BezrukovKuzmanovskiLim2023} for block domination orders and \cite{AhlswedeBezrukov1995, BollobasLeader1991_grid} for the the order on the product of trees.
	We suspect that slight modifications of these orders will produce all possible compressed optimal orders.
	\item Prove a local-global principle for the orders in step \ref{delta=1}.
	We suspect that pull-push method developed in \cite{BezrukovKuzmanovskiLim2023} can be used to solve this problem as well.
	\item\label{delta=0} Find all optimal orders when $\delta(i) \geq 0$.
	We believe that a generalization of the orders in step \ref{delta=1} using block orders from \cite{BezrukovKuzmanovskiLim2023} will work and maybe some slight modification will be needed.
	\item Prove a local-global principle for all orders in step \ref{delta=0}. 
	Again, we suspect that the pull-push method from \cite{BezrukovKuzmanovskiLim2023} can be used here again.
\end{enumerate}

\section{Acknowledgments}
The author would like to thank Sergei L. Bezrukov and Jamie Radcliffe for reading early drafts of this paper.
The proof of Theorem \ref{uniqueness} was simplified drastically after discussions with them.

\bibliographystyle{acm}

\end{document}